\documentclass{amsart}
\usepackage[active]{srcltx}
\usepackage{calc,amssymb,amsthm,amsmath,amscd, eucal,ulem}
\usepackage{alltt}
\usepackage[left=1.35in,top=1.25in,right=1.35in,bottom=1.25in]{geometry}
\RequirePackage[dvipsnames,usenames]{color}

\input{kmacros3.sty}
\input{xy}
\xyoption{all}
\usepackage{tikz}

\numberwithin{equation}{theorem}

\newcommand{\D}{\displaystyle}

\DeclareMathOperator{\E}{E}

\renewcommand{\m}{\mathfrak{m}}

\usepackage{fullpage}

\usepackage{setspace}
\usepackage{hyperref}

\usepackage{enumerate}

\usepackage{graphicx}

\usepackage[all,cmtip]{xy}
\usepackage{verbatim}

\theoremstyle{theorem}

\begin{document}
\title{Eulerian graded $\scr{D}$-modules}
\author{Linquan Ma and Wenliang Zhang}
\address{Department of Mathematics\\ University of Michigan\\ Ann Arbor\\ Michigan 48109}
\email{lquanma@umich.edu}
\address{Department of Mathematics\\ University of Nebraska\\ Lincoln\\ Nebraska 68588}
\email{wzhang15@unl.edu}
\thanks{The second author was partially supported by the NSF grant DMS \#1068946.}
\maketitle

\begin{abstract}
Let $R=K[x_1,\dots,x_n]$ with $K$ a field of arbitrary characteristic and
$\scr{D}$ be the ring of differential operators over $R$. Inspired by Euler formula for
homogeneous polynomials, we introduce a class of graded
$\scr{D}$-modules, called {\it Eulerian} graded $\scr{D}$-modules.
It is proved that a vast class of $\scr{D}$-modules, including all
local cohomology modules $H^{i_1}_{J_1}\cdots H^{i_s}_{J_s}(R)$
where $J_1,\dots,J_s$ are homogeneous ideals of $R$, are Eulerian.
As an application of our theory of Eulerian graded
$\scr{D}$-modules, we prove that all socle elements of each local
cohomology module $H^{i_0}_{\bm}H^{i_1}_{J_1}\cdots
H^{i_s}_{J_s}(R)$ must be in degree $-n$ in all characteristic. This answers a question raised in
\cite{YiZhangGradedFModules}. It is also proved that graded $F$-modules are Eulerian and hence the main result in \cite{YiZhangGradedFModules} is recovered. An application of our theory of Eulerian graded $\scr{D}$-modules to the graded injective hull of $R/P$, where $P$ is a homogeneous prime ideal of $R$, is discussed as well.
\end{abstract}

\section{Introduction}
Let $R=K[x_1,\dots,x_n]$ be a polynomial ring in $n$ indeterminates
over a field $K$ with the standard grading, {\it
i.e.} $\deg(x_i)=1$ for each $x_i$ and $\deg(c)=0$ for each nonzero
$c\in K$. Let $\partial_i^{[j]}$ denote the $j$-th order differential operator
$\D\frac{1}{j!}\cdot\frac{\partial^j}{\partial x_i^j}$ with respect to $x_i$ for each $0\leq i\leq n$ and $j\geq 1$.
By \cite[Th\'eor\`eme~16.11.2]{EGAIV}, $\scr{D}=R\langle \partial_i^{[j]}| 1\leq i\leq n, 1\leq j\rangle$ is the ring of $K$-linear differential operators of $R$ (note if $K$ has characteristic $0$, $\scr{D}$ is the same as the Weyl algebra $R\langle\partial_1,\dots,\partial_n\rangle$). The ring of $K$-linear differential operators $\scr{D}$ has a natural $\mathbb{Z}$-grading given by $\deg(x_i)=1$, $\deg(\partial_i^{[j]})=-j$, and $\deg(c)=0$ for each $x_i$, $\partial_i^{[j]}$ and each nonzero $c\in K$.

The classical Euler formula for homogeneous polynomials says that
\[\sum^n_{i=1}x_i\partial_i f =\deg(f)f\]
for each homogeneous polynomial $f\in R$. Inspired by Euler formula, we introduce a class of $\scr{D}$-modules called Eulerian graded $\scr{D}$-modules: the graded $\scr{D}$-modules whose homogeneous elements satisfy a series of ``higher order Euler formulas" ({\it cf.} Definition \ref{Eulerian D-modules}). One of our main results concerning Eulerian graded $\scr{D}$-modules is the following (proved in Section 2 and Section 5):
\begin{theorem}
Let $R=K[x_1,\dots,x_n]$, $\bm=(x_1,\dots,x_n)$, and $J_1,\dots, J_s$ be homogeneous ideals of $R$. Then
\begin{enumerate}
\item $R(\ell)$ is Eulerian if and only if $\ell=0$.
\item Let $\sideset{^*}{}\E$ be the graded injective hull of $R/\bm$. Then $\sideset{^*}{}\E(\ell)$ is Eulerian if and only if $\ell=n$.
\item Each local cohomology module $H^{i_1}_{J_1}(\cdots(H^{i_s}_{J_s}(R)))$ is Eulerian for all $i_1,\dots, i_s$.
\end{enumerate}
\end{theorem}

As an application of our theory of Eulerian graded $\scr{D}$-modules, we have the following result on local cohomology (proved in Section 5):
\begin{theorem}
\label{main theorem}
Let notations be as in the previous theorem. Then all socle elements of each $H^{i_0}_{\bm}(H^{i_1}_{J_1}(\cdots(H^{i_s}_{J_s}(R)))$ must have degree $-n$, and consequently each $H^{i_0}_{\bm}(H^{i_1}_{J_1}(\cdots(H^{i_s}_{J_s}(R)))$ is isomorphic (as a graded $\scr{D}$-module) to a direct sum of copies of $\sideset{^*}{}\E(n)$.
\end{theorem}

This result is characteristic-free, in particular it gives a positive answer to a question stated in \cite{YiZhangGradedFModules} and recovers the main theorem in \cite{YiZhangGradedFModules}.

The paper is organized as follows. In Section 2, Eulerian graded $\scr{D}$-modules are defined over an arbitrary field and some basic properties of these modules are discussed. In Section 3 and Section 4, we consider Eulerian graded $\scr{D}$-modules in characteristic 0 and characteristic $p$, respectively; in particular, we show in Section 4 that each graded $F$-module (introduced in \cite{YiZhangGradedFModules}) is Eulerian. In Section 5, we apply our theory of Eulerian graded $\scr{D}$-modules to local cohomology modules; Theorem \ref{main theorem} is proved in this section. Finally, in Section 6, an application of our theory to graded injecitve hull of $R/P$, where $P$ is a homogeneous prime ideal, is considered.

We finish our introduction by fixing our notation throughout the paper as follows. $R=K[x_1,\dots,x_n]$ denotes the polynomial ring in $n$ indeterminates over a field $K$.  The $j$-th order differential operator $\D\frac{1}{j!}\cdot\frac{\partial^j}{\partial x_i^j}$ with respect to $x_i$  is denoted by $\partial_i^{[j]}$ and $\scr{D}=R\langle \partial_i^{[j]}| 1\leq i\leq n, 1\leq j\rangle$ denotes the ring of differential operators over $R$. It follows from \cite[Corollary 2.2]{LyubeznikCharFreeDModules} that every element of $\scr{D}$ may be uniquely written as a linear combinations of monomials in $x$'s and $\partial$'s. The natural $\mathbb{Z}$-grading on $R$ and $\scr{D}$ is given by
\[\deg(x_i)=1,\ \deg(\partial_i^{[j]})=-j,\ \deg(c)=0\]
for each $x_i$, $\partial_i^{[j]}$ and nonzero $c\in K$ (it is evident that $R$ is a graded $\scr{D}$-module).

A graded $\scr{D}$-module is a left $\scr{D}$-module with a $\mathbb{Z}$-grading that is compatible with the natural $\mathbb{Z}$-grading on $\scr{D}$. Given any graded $\scr{D}$-module $M$, the module $M(\ell)$ denotes $M$ with degree shifted by $\ell$, {\it i.e.} $M(\ell)_i=M_{\ell+i}$ for each $i$.

The irrelevant maximal ideal $(x_1,\dots,x_n)$ of $R$ is denoted by
$\bm$. The graded injective hull of $R/\bm$ is denoted by
$\sideset{^*}{}\E$. It may be identified with the $K$-vector space with a basis $\{\frac{1}{x^{e_1}_1\cdots x^{e_n}_n}\vert e_1,\dots,e_n\geq 1\}$ and it has a natural $\scr{D}$-module structure given by
\[\partial_j\cdot \frac{1}{x^{e_1}_1\cdots x^{e_n}_n} =  \frac{-e_j}{x^{e_1}_1\cdots x^{e_j+1} \cdots x^{e_n}_n}. \]
$\sideset{^*}{}\E$ is graded to the effect that the element
$\D\frac{1}{x_1\cdots x_n}$ has degree 0 ({\it cf.} \cite[Example
13.3.9]{BrodmannSharpLocalCohomology}).

For each integer $a$ and a nonnegative integer $b$, we will use $\D\binom{a}{b}$ to denote $\D\frac{a\cdot(a-1)\cdots(a-b+1)}{b!}$ (note that this number is still well-defined when $\Char(K)=p>0$).

\section*{Acknowledgement} The authors would like to thank Professor Gennady Lyubeznik for carefully reading a preliminary version of the paper and for his valuable suggestions which improved the paper considerably. The authors are also grateful to the anonymous referee for her/his comments on the paper.

\section{Eulerian graded $\scr{D}$-modules}
In this section, we introduce Eulerian graded $\scr{D}$-modules and discuss some of their basic properties. We begin with the following definition.

\begin{definition}
\label{Eulerian D-modules}
The $r$-th Euler operator, denoted by $E_r$, is defined as
\[E_r:=\D\sum_{i_1+i_2+\cdots+i_n=r,i_1\geq 0,\dots,i_n\geq 0}x_1^{i_1}\cdots x_n^{i_n}\partial_1^{[i_1]}\cdots\partial_n^{[i_n]}.\]
In particular $E_1$ is the usual Euler operator $\sum_{i=1}^nx_i\partial_i$.

A graded $\scr{D}$-module $M$ is called {\it Eulerian}, if each homogeneous element $z\in M$ satisfies
\begin{equation}
\label{Euler's formula}
E_r\cdot z=\binom{\deg(z)}{r}\cdot z
\end{equation}
for every $r\geq 1$.
\end{definition}

We start with an easy lemma.
\begin{lemma}
\label{combine two operators}
For all positive integers $s$ and $t$, we have $\D\partial^{[s]}_i\partial^{[t]}_i=\binom{s+t}{s}\partial^{[s+t]}_i$.
\end{lemma}
\begin{proof}
It is easy to check that (in all characteristic) \[\partial^{[s]}_i\partial^{[t]}_i=\frac{\partial^s_i}{s!}\frac{\partial^t_i}{t!}=\binom{s+t}{s}\frac{\partial^{s+t}_i}{(s+t)!}=\binom{s+t}{s}\partial^{[s+t]}_i.\]
\end{proof}

The following lemma is a special case of \cite[Proposition 2.1]{LyubeznikCharFreeDModules}, which will be needed in the sequel.

\begin{lemma}
\label{swapping x and d}
For all positive integers $s$ and $t$, we have
\[\partial^{[s]}_ix^t_i=\sum_{j=0}^{\min\{s,t\}}\binom{t}{j}x_i^{t-j}\partial^{[s-j]}_i\]
for each $i$.
\end{lemma}

The following proposition indicates a connection among Euler operators.
\begin{proposition}
\label{relation about E_r and E_r+1}
For every $r\geq 1$, we have $E_1\cdot E_r=(r+1)E_{r+1}+rE_r$.
\end{proposition}
\begin{proof}
By Lemma \ref{combine two operators} we know $\partial_i\partial_i^{[j]}=(j+1)\partial_i^{[j+1]}$. Now we have

\begin{align}
E_1\cdot E_r &=\sum_jx_j\partial_j\cdot\sum_{i_1+i_2+\cdots+i_n=r}x_1^{i_1}x_2^{i_2}\cdots x_n^{i_n}\partial_1^{[i_1]}\partial_2^{[i_2]}\cdots\partial_n^{[i_n]}\notag\\
&=\sum_j\sum_{i_1+i_2+\cdots+i_n=r}x_1^{i_1}\cdots(x_j\partial_j x_j^{i_j})\cdots x_n^{i_n}\partial_1^{[i_1]}\partial_2^{[i_2]}\cdots\partial_n^{[i_n]}\notag\\
&=\sum_j\sum_{i_1+i_2+\cdots+i_n=r}x_1^{i_i}\cdots(x_j^{i_j+1}\partial_j+i_jx_j^{i_j})\cdots x_n^{i_n}\partial_1^{[i_1]}\partial_2^{[i_2]}\cdots\partial_n^{[i_n]}\notag\ {\rm by\ lemma\ \ref{swapping x and d}}\\
&=\sum_{i_1+i_2+\cdots+i_n=r}\sum_jx_1^{i_1}\cdots x_j^{i_j+1}\cdots x_n^{i_n}\partial_1^{[i_1]}\cdots(\partial_j\partial_j^{[i_j]})\cdots\partial_n^{[i_n]}\notag\\
&\quad +\sum_{i_1+i_2+\cdots+i_n=r}\sum_ji_jx_1^{i_1}\cdots x_n^{i_n}\partial_1^{[i_1]}\cdots\partial_n^{[i_n]}\notag\\
&=\sum_{i_1+i_2+\cdots+i_n=r}\sum_j(i_j+1)x_1^{i_1}\cdots x_j^{i_j+1}\cdots x_n^{i_n}\partial_1^{[i_1]}\cdots\partial_j^{[i_j+1]}\cdots\partial_n^{[i_n]}+ rE_r\notag\\
&=(r+1)E_{r+1}+rE_r\notag
\end{align}
\end{proof}

Some remarks are in order.

\begin{remark}
\label{basic remarks}
\begin{enumerate}
\item If $M$ is an Eulerian graded $\scr{D}$-module, then $M(\ell)$ is Eulerian if and only if $\ell=0$ (equivalently, for any graded $\scr{D}$-module $M$, $M(\ell)$ is Eulerian graded for at most one $\ell$). Because in any characteristic, we have $\D\binom{a}{r}=\binom{b}{r}$ for all $r\in \mathbb{N}$ if and only if $a=b$. This is trivial in characteristic $0$ (take $r=1$). We give a short argument in characteristic $p>0$. If both $a$ and $b$ are positive, then by a theorem of Lucas \cite{Lucascombinatorialnumber}, we have \[\D\binom{a}{r}=\prod\binom{a_i}{r_i}\] where $a=\sum a_ip^i$ and $r=\sum r_ip^i$ are the $p$-adic decompositions of $a$ and $r$. In particular, if we apply $\D\binom{a}{r}=\binom{b}{r}$ to $r=p^i$, we get $a_i=b_i$ for every $i$, hence $a=b$. When $a$ and $b$ are both negative, we look at the $p$-adic decomposition of $-a-1=\sum a_ip^i$ and $-b-1=\sum b_ip^i$. Taking $r=p^i$, we get \[(-1)^{p^i}\cdot(a_i+1)=(-1)^{p^i}\binom{-a+p^i-1}{p^i}=\binom{a}{p^i}=\binom{b}{p^i}=\binom{-b+p^i-1}{p^i}=(-1)^{p^i}\cdot(b_i+1).\] Hence again we have $a_i=b_i$ for every $i$, $-a-1=-b-1$ so $a=b$. Therefore it suffices to show $a$ and $b$ must have the same sign. But if $a>0$ and $b<0$, we can pick $r=p^j\gg a-b$, then direct computation (or using the theorem of Lucas \cite{Lucascombinatorialnumber}) gives $\D\binom{a}{p^j}=0$ while $\D\binom{b}{p^j}=(-1)^{p^j}\binom{-b+p^j-1}{p^j}=(-1)^{p^j}$, which is a contradiction.

\item Definition \ref{Eulerian D-modules} does not depend on the characteristic of $K$. However we will see in section 3 that, in characteristic $0$, we only need to consider $E_1$, the usual Euler operator.
\item $R$ is Eulerian and a proof goes as follows. For each monomial $x_1^{j_1}x_2^{j_2}\cdots x_n^{j_n}$, where $j_n$'s are arbitrary integers (we allow negative integers), we have (for each $r\geq 1$)
\begin{eqnarray*}
E_r\cdot x_1^{j_1}x_2^{j_2}\cdots x_n^{j_n}&=& (\sum_{i_1+i_2+\cdots+i_n=r}x_1^{i_1}x_2^{i_2}\cdots x_n^{i_n}\partial_1^{[i_1]}\partial_2^{[i_2]}\cdots\partial_n^{[i_n]})\cdot(x_1^{j_1}x_2^{j_2}\cdots x_n^{j_n})\\ &=&\sum_{i_1+i_2+\cdots+i_n=r}\binom{j_1}{i_1}\cdots\binom{j_n}{i_n}\cdot x_1^{j_1}x_2^{j_2}\cdots x_n^{j_n}\\
&=&\binom{j_1+\cdots+j_n}{r}\cdot x_1^{j_1}x_2^{j_2}\cdots x_n^{j_n}
\end{eqnarray*}
We explain the last equality: we use an induction argument. It is clear when $r=1$, now suppose we have the equality for $r-1$. Then
\begin{eqnarray*}
&&\sum_{i_1+i_2+\cdots+i_n=r}\binom{j_1}{i_1}\cdots\binom{j_n}{i_n}\\
&=&\sum_{i_n=0}^r\binom{j_n}{i_n}\sum_{i_1+i_2+\cdots+i_{n-1}=r-i_n}\binom{j_1}{i_1}\cdots\binom{j_{n-1}}{i_{n-1}}\\
&=&\sum_{i_n=0}^r\binom{j_n}{i_n}\binom{j_1+\cdots+j_{n-1}}{r-i_n}\\
&=&\binom{j_1+\cdots+j_n}{r}
\end{eqnarray*}
where the second equality is by induction hypothesis for $r-1$, and the last equality is by the Chu-Vandermonde identity (see \cite{AskeyOrthogonalPolynomial}) \[\D\binom{a}{r}=\sum_{k=0}^r\binom{j}{k}\binom{a-j}{r-k}.\]
Since $x_1^{j_1}x_2^{j_2}\cdots x_n^{j_n}$ clearly has degree $j_1+\cdots+j_n$ in $R$ and the $E_r$'s clearly preserve addition, we can see from the computation above that $\D E_r\cdot z=\binom{\deg(z)}{r}\cdot z$ for every homogenous $z\in R$. Therefore, $R$ is Eulerian.
\item $\scr{D}$ is {\it not} Eulerian. It is clear that the identity $1\in \scr{D}$ is homogeneous with degree 0, but $E_1\cdot 1=(\sum^n_{i=1}x_i\partial_i)\cdot 1=\sum^n_{i=1}x_i\partial_i\neq 0$.
\end{enumerate}
\end{remark}

One of the main results in this section is that, to check whether a graded $\scr{D}$-module is Eulerian, it suffices to check whether each element of any set of homogeneous generators satisfies (\ref{Euler's formula}).
\begin{theorem}
\label{Eulerian is detemined by generators}
Let $M$ be a graded $\scr{D}$-module. Assume that $\{g_1,g_2,\dots\}$ is a set of homogeneous $\scr{D}$-generators of $M$. Then, $M$ is Eulerian if and only if each $g_j$ satisfies
Euler formula (\ref{Euler's formula}) for every $r$.
\end{theorem}
\begin{proof}
If $M$ is Eulerian, then it is clear that each $g_j$ satisfies Euler formula (\ref{Euler's formula}) for every $r$. Assume that each $g_j$ satisfies Euler formula (\ref{Euler's formula}) for every $r$ and we wish to prove that $M$ is Eulerian. To this end, it suffices to show that, if a homogeneous element $z\in M$ satisfies Euler formula (\ref{Euler's formula}) for every $r$, then so does $x^{s_1}_1\cdots x^{s_n}_n\partial^{[j_1]}_1\cdots \partial^{[j_n]}_n\cdot z$. And it is clear that it suffices to consider $x^{s}_i\partial^{[j]}_i\cdot z$. Without loss of generality, we may assume $i=1$. We will prove this in two steps; first we consider $\partial^{[j]}_1\cdot z$ and then $x^s_1\cdot z$ (once we finish our first step, we may replace $\partial^{[j]}\cdot z$ by $z$ and then our second step will finish the proof).

First we will use induction on $r$ to show that $\partial^{[j]}_1\cdot z$ satisfies Euler formula (\ref{Euler's formula}) for each $r$. When $r=1$, we compute
\begin{align}
E_1\cdot(\partial_1^{[j]}z)&=\sum_{i=1}^nx_i\partial_i\cdot(\partial_1^{[j]}z)\notag\\
&=x_1\partial_1^{[j]}\partial_1z+\partial_1^{[j]}\sum_{i\geq 2}x_i\partial_i\cdot z\notag\\
&=\partial_1^{[j]}x_1\partial_1z-\partial_1^{[j-1]}\partial_1z+\partial_1^{[j]}\sum_{i\geq 2}x_i\partial_i\cdot z\notag\\
&=\partial_1^{[j]}\sum_{i=1}^nx_i\partial_i\cdot z-j\partial_1^{[j]}z\notag\\
&=(\deg(z)-j)\cdot\partial_1^{[j]}z\notag
\end{align}
Now for general $r$, suppose we know that $E_{r-k}\cdot(\partial_1^{[j]}z)=\D\binom{\deg(z)-j}{r-k}\cdot \partial_1^{[j]}z$ for every $1\leq k\leq r-1$. Then we have
\begin{align}
&\qquad(\sum_{i_1+i_2+\cdots+i_n=r}x_1^{i_1}x_2^{i_2}\cdots x_n^{i_n}\partial_1^{[i_1]}\partial_2^{[i_2]}\cdots\partial_n^{[i_n]})\cdot
(\partial_1^{[j]}z)\notag\\
&=\sum_{i_1+i_2+\cdots+i_n=r}(x_1^{i_1}\partial_1^{[j]})x_2^{i_2}\cdots x_n^{i_n}\partial_1^{[i_1]}\partial_2^{[i_2]}\cdots\partial_n^{[i_n]}\cdot
z\notag\\
&=\sum_{i_1+i_2+\cdots+i_n=r}(\partial_1^{[j]}x_1^{i_1})x_2^{i_2}\cdots x_n^{i_n}\partial_1^{[i_1]}\partial_2^{[i_2]}\cdots\partial_n^{[i_n]}\cdot z\tag{1}
\\
&\qquad -\sum_{i_1+i_2+\cdots+i_n=r}\sum_{k=1}^{\min\{i_1,j\}}\binom{i_1}{k}x_1^{i_1-k}\partial_1^{[j-k]}x_2^{i_2}\cdots x_n^{i_n}\partial_1^{[i_1]}\partial_2^{[i_2]}\cdots\partial_n^{[i_n]}\cdot z\notag\\
&=\binom{\deg(z)}{r}\cdot(\partial_1^{[j]}z)-\sum_{i_1+i_2+\cdots+i_n=r}\sum_{k=1}^{\min\{i_1,j\}}x_1^{i_1-k}x_2^{i_2}\cdots x_n^{i_n}(\binom{i_1}{k}\partial_1^{[i_1]})\partial_2^{[i_2]}\cdots\partial_n^{[i_n]}\cdot (\partial_1^{[j-k]}z)\notag\\
&=\binom{\deg(z)}{r}\cdot(\partial_1^{[j]}z)-\sum_{i_1+i_2+\cdots+i_n=r}\sum_{k=1}^{\min\{i_1,j\}}x_1^{i_1-k}x_2^{i_2}\cdots x_n^{i_n}\partial_1^{[i_1-k]}\partial_2^{[i_2]}\cdots\partial_n^{[i_n]}\cdot (\partial_1^{[k]}\partial_1^{[j-k]}z)\tag{2}
\end{align}

\begin{align}
&=\binom{\deg(z)}{r}\cdot(\partial_1^{[j]}z)-\sum_{i_1+i_2+\cdots+i_n=r}\sum_{k=1}^{\min\{i_1,j\}}x_1^{i_1-k}x_2^{i_2}\cdots x_n^{i_n}\partial_1^{[i_1-k]}\partial_2^{[i_2]}\cdots\partial_n^{[i_n]}\cdot \binom{j}{k}(\partial_1^{[j]}z)\tag{3}\\
&=\binom{\deg(z)}{r}\cdot(\partial_1^{[j]}z)-\sum_{k=1}^j\binom{j}{k}\sum_{i_1'+i_2+\cdots+i_n=r-k}x_1^{i_1'}x_2^{i_2}\cdots x_n^{i_n}\partial_1^{[i_1']}\partial_2^{[i_2]}\cdots\partial_n^{[i_n]}\cdot(\partial_1^{[j]}z)\tag{4}\\
&=\binom{\deg(z)}{r}\cdot(\partial_1^{[j]}z)-\sum_{k=1}^j\binom{j}{k}\binom{\deg(z)-j}{r-k}\cdot(\partial_1^{[j]}z)\tag{5}\\
&=\binom{\deg(z)-j}{r}\cdot(\partial_1^{[j]}z)\tag{6}\\
&=\binom{\deg(\partial_1^{[j]}z)}{r}\cdot(\partial_1^{[j]}z)\notag
\end{align}
where (1) follows from Lemma \ref{swapping x and d}:
\[x_1^{i_1}\partial_1^{[j]}=\partial_1^{[j]}x_1^{i_1}-\sum_{k=1}^{\min\{i_1,j\}}\binom{i_1}{k}x_1^{i_1-k}\partial_1^{[j-k]};\]
(2) and (3) follow from Lemma \ref{combine two operators}; \\
(4) is obtained by setting $i_1'=i_1-k$;\\
(5) is true by induction on $r$;\\
(6) follows from the Chu-Vandermonde identity (\cite[page 59]{AskeyOrthogonalPolynomial}) \[\D\binom{a}{r}=\sum_{k=0}^j\binom{j}{k}\binom{a-j}{r-k}.\]
This completes our first step.

Next we consider $x_1\cdot z$ and we have
\begin{align}
&\qquad (\sum_{i_1+i_2+\cdots+i_n=r}x_1^{i_1}x_2^{i_2}\cdots x_n^{i_n}\partial_1^{[i_1]}\partial_2^{[i_2]}\cdots\partial_n^{[i_n]})\cdot (x_1\cdot z)\notag\\
&=(\sum_{i_1+i_2+\cdots+i_n=r}x_1^{i_1}x_2^{i_2}\cdots x_n^{i_n}(\partial_1^{[i_1]}x_1)\partial_2^{[i_2]}\cdots\partial_n^{[i_n]})\cdot z\notag\\
&=(\sum_{i_1+i_2+\cdots+i_n=r,i_1\geq 1}x_1^{i_1}x_2^{i_2}\cdots x_n^{i_n}(x_1\partial_1^{[i_1]}+\partial^{[i_1-1]}_1) \partial_2^{[i_2]}\cdots\partial_n^{[i_n]}+x_1\sum_{i_2+\cdots+i_n=r}x_2^{i_2}\cdots x_n^{i_n}\partial_2^{[i_2]}\cdots\partial_n^{[i_n]})\cdot z\notag\\
&= x_1(\sum_{i_1+i_2+\cdots+i_n=r}x_1^{i_1}x_2^{i_2}\cdots x_n^{i_n}\partial_1^{[i_1]} \partial_2^{[i_2]}\cdots\partial_n^{[i_n]}+\sum_{i_1+i_2+\cdots+i_n=r,i_1\geq 1}x^{i_1-1}_1x^{i_2}_2\cdots x^{i_n}_n\partial^{[i_1-1]}_1\partial^{[i_2]}_2\cdots \partial^{[i_n]}_n)\cdot z\notag \\
&= x_1(\binom{\deg(z)}{r}+\binom{\deg(z)}{r-1})z\notag\\
&= \binom{\deg(z)+1}{r}(x_1\cdot z)\notag\\
&= \binom{\deg(x_1\cdot z)}{r}(x_1\cdot z)\notag
\end{align}
This finishes our second step in the case when $s=1$.

Now we consider $x^s_1\cdot z$ when $s\geq 2$. By an easy induction we may assume $x^{s-1}_1z$ satisfies (\ref{Euler's formula}), now we have
\[E_r(x^{s}_1z)=E_r(x_1x^{s-1}_1z)=\binom{\deg(x_1(x^{s-1}z))}{r}(x_1(x^{s-1}z))=\binom{\deg(x^s_1z)}{r}(x^s_1z).\]
where the second equality is the case when $s=1$ (because we assume $x^{s-1}_1z$ satisfies (\ref{Euler's formula})). This completes the proof of our theorem.
\end{proof}

An immediate consequence of Theorem \ref{Eulerian is detemined by generators} on cyclic $\scr{D}$-modules is the following.
\begin{proposition}
\label{D/J is Eulerian iff E_r in J}
Let $J$ be a homogeneous left ideal in $\scr{D}$. Then
$\D\frac{\scr{D}}{J}$ is Eulerian if and only if
$\D E_r=\sum_{i_1+i_2+\cdots+i_n=r}x_1^{i_1}x_2^{i_2}\cdots x_n^{i_n}\partial_1^{[i_1]}\partial_2^{[i_2]}\cdots\partial_n^{[i_n]} \in J$ for every $r$.
\end{proposition}
\begin{proof}
According to Theorem \ref{Eulerian is detemined by generators}, $\D\frac{\scr{D}}{J}$ is
Eulerian if and only if $\D\overline{1}\in \frac{\scr{D}}{J}$
satisfies Euler's formula (\ref{Euler's formula}). Since
$\overline{1}$ has degree 0, $\overline{1}\in \D\frac{\scr{D}}{J}$
satisfies Euler's formula (\ref{Euler's formula}) if and only if
$E_r\cdot
\overline{1}=\overline{E_r}=0\in\D
\frac{\scr{D}}{J}$, which holds if and only if
$E_r \in J$ for every $r$.
\end{proof}

\begin{proposition}
\label{submodules and quotients}
If a graded $\scr{D}$-module $M$ is Eulerian, so are each graded submodule of $M$ and each graded quotient of $M$.
\end{proposition}
\begin{proof}
Let $N$ be a graded submodule of $M$. Since each homogeneous element is also a homogeneous element in $M$, it is clear that $N$ is also Eulerian. Given a $\scr{D}$-linear degree-preserving surjection $\psi:M\to M'$ and a homogeneous element $z'\in M'$, there is a homogeneous $z\in M$ with the same degree such that $\psi(z)=z'$ and hence we have (for every $r$)
\[E_r\cdot z'= E_r\cdot\psi(z)=\psi(E_r\cdot z)=\psi(\binom{\deg(z)}{r}\cdot z)=\binom{\deg(z')}{r} \cdot z'.\]
This proves that $M'$ is also Eulerian.
\end{proof}

We end this section with the following result which is one of the key ingredients for our application to local cohomology.
\begin{theorem}
\label{R and E}
\begin{enumerate}
\item The graded $\scr{D}$-module $R(\ell)$ is Eulerian if and only if $\ell=0$.
\item The graded $\scr{D}$-module $\sideset{^*}{}\E(\ell)$ is Eulerian if and only $\ell=n$.
\end{enumerate}
\end{theorem}
\begin{proof}
By Remark \ref{basic remarks} (1), it suffices to show that $R(0)$ and $\sideset{^*}{}\E(n)$ are Eulerian graded. It is clear that $R=R(0)$ is Eulerian by Remark \ref{basic remarks} (3). Since $\sideset{^*}{}\E(\ell)$ is spanned over $K$ by $x_1^{j_1}x_2^{j_2}\cdots x_n^{j_n}$ with each $j_t\leq -1$. By the computation in Remark \ref{basic remarks} (3), it is clear that $\sideset{^*}{}\E(n)$ is Eulerian graded (because in $\sideset{^*}{}\E(n)$, the element $x_1^{j_1}x_2^{j_2}\cdots x_n^{j_n}$ has degree $j_1+\cdots+j_n$).
\end{proof}

\section{Eulerian graded $\scr{D}$-module in characteristic $0$}
Throughout this section $K$ will be a field of characteristic $0$. In this section we collect some properties of Eulerian graded $\scr{D}$-modules when $\Char(K)=0$. The main result is that, if a graded $\scr{D}$-module $M$ is Eulerian, then so is $M_f$ for each $f\in R$. This is one of the ingredients for our application to local cohomology in Section \ref{local cohomology}. First we observe that, in characteristic 0, if each homogeneous element $z$ in a graded $\scr{D}$-module $M$ satisfies (\ref{Euler's formula}) for $r=1$ (instead of for all $r\geq 1$), then $M$ is Eulerian.

\begin{proposition}
\label{Eulerian in char 0 is determined by E_1}
Let $M$ be a graded $\scr{D}$-module. If $E_1\cdot z=\deg(z)\cdot z$ for every homogeneous element $z\in M$. Then $M$ is Eulerian.
\end{proposition}
\begin{proof}
We prove by induction that $E_r\cdot z=\D\binom{\deg(z)}{r}\cdot z$ for every $r\geq 1$. When $r=1$ this is exactly $E_1\cdot z=\deg(z)\cdot z$ which is given. Now suppose we know $\D E_r\cdot z=\binom{\deg(z)}{r}\cdot z$. By Proposition \ref{relation about E_r and E_r+1}, we know that $E_{r+1}=\D\frac{1}{r+1}(E_1\cdot E_r-rE_r)$. So we have
\begin{eqnarray*}
E_{r+1}\cdot z&=&\frac{1}{r+1}(E_1\cdot E_r-rE_r)\cdot z \\
&=&\frac{1}{r+1}(E_1\cdot\binom{\deg(z)}{r}\cdot z-r\cdot\binom{\deg(z)}{r}\cdot z)\\
&=&\frac{1}{r+1}\cdot\binom{\deg(z)}{r}\cdot(\deg(z)\cdot z-rz)\\
&=&\frac{1}{r+1}\cdot\binom{\deg(z)}{r}(\deg(z)-r)\cdot z\\
&=&\binom{\deg(z)}{r+1}\cdot z
\end{eqnarray*}
where the second equality uses the induction hypothesis. This finishes the proof.
\end{proof}

\begin{remark}
As we have seen, Lemma \ref{relation about E_r and E_r+1} is quite useful when $\Char(K)=0$. Unfortunately, this is no longer the case once we are in characteristic $p$. For instance, when $r\equiv -1$(mod $p$), we can't link $E_{r+1}$ and $E_r$ via Lemma \ref{relation about E_r and E_r+1}. This is one of the reasons that we treat characteristic 0 and characteristic $p$ separately in two different sections.
\end{remark}

\begin{corollary}
Let $J$ be a homogeneous left ideal in $\scr{D}$. Then $\D\frac{\scr{D}}{J}$ is Eulerian if and only if
$\D\sum^n_{i=1}x_i\partial_i \in J$.
\end{corollary}
\begin{proof}
This is clear from Proposition \ref{Eulerian in char 0 is determined by E_1} and (the proof of) Proposition \ref{D/J is Eulerian iff E_r in J}.
\end{proof}

\begin{proposition}
\label{localization of Eulerian D-modules in char 0} If $M$ is an Eulerian
graded $\scr{D}$-module, so is $S^{-1}M$ for each homogeneous multiplicative system $S\subseteq R$. In particular, $M_f$ is Eulerian for
each homogeneous polynomial $f \in R$.
\end{proposition}
\begin{proof}
By Proposition \ref{Eulerian in char 0 is determined by E_1}, it suffices to show for each homogeneous $f\in S$ and $z\in M$, we have $E_1\cdot\D\frac{z}{f^t}=\deg(\frac{z}{f^t})\cdot \frac{z}{f^t}$. Now we compute
\begin{eqnarray*}
E_1\cdot \frac{z}{f^t}&=&\sum_{i=1}^nx_i\partial_i\cdot \frac{z}{f^t}\\
&=&\sum_{i=1}^nx_i\cdot \frac{f^t\cdot
\partial_i(z)-\partial_i(f^t)\cdot z}{f^{2t}}\\
&=&\frac{1}{f^{2t}}(f^t\sum_{i=1}^nx_i\partial_i(z)-
\sum_{i=1}^nx_i\partial_i(f^t)\cdot z)\\
&=&\frac{1}{f^t}\cdot \deg(z)\cdot z- \frac{1}{f^{2t}}\cdot
\deg(f^t)\cdot f^t\cdot z\\
&=&(\deg(z)-\deg(f^t))\cdot \frac{z}{f^t}\\
&=&\deg(\frac{z}{f^t})\cdot\frac{z}{f^t}
\end{eqnarray*}
This finishes the proof.
\end{proof}

\begin{remark}
\begin{enumerate}
\item It turns out that Eulerian graded $\scr{D}$-modules are {\it not} stable under extension because of the following short exact sequence of graded $\scr{D}$-modules:
\[0\to \frac{\scr{D}}{\langle \sum^n_{i=1}x_i\partial_i \rangle}\xrightarrow{\sum^n_{i=1}x_i\partial_i} \frac{\scr{D}}{\langle (\sum^n_{i=1}x_i\partial_i)^2 \rangle}\to \frac{\scr{D}}{\langle \sum^n_{i=1}x_i\partial_i \rangle}\to 0,\]
where the map $\D\frac{\scr{D}}{\langle \sum^n_{i=1}x_i\partial_i \rangle}\xrightarrow{\sum^n_{i=1}x_i\partial_i} \frac{\scr{D}}{\langle (\sum^n_{i=1}x_i\partial_i)^2 \rangle}$ is the multiplication by $\sum^n_{i=1}x_i\partial_i$, {\it i.e.} $\overline{a}\mapsto \overline{a\cdot (\sum^n_{i=1}x_i\partial_i) }$.
\item Since $\D\dim( \frac{\scr{D}}{\langle \sum^n_{i=1}x_i\partial_i \rangle})=2n-1$ and $ \D\frac{\scr{D}}{\langle \sum^n_{i=1}x_i\partial_i \rangle}$ is Eulerian, finitely generated (even cyclic) Eulerian graded $\scr{D}$-modules may not be holonomic when $n\geq 2$.
\item When $n=1$, it is rather straightforward to check that each finitely generated Eulerian graded $\scr{D}$-module is holonomic.
\item As we will see in section 5, in characteristic $0$, a vast class of graded $\scr{D}$-modules (namely local cohomology modules of $R$) are both Eulerian and holonomic.
\end{enumerate}
\end{remark}

\section{Eulerian graded $\scr{D}$-module in characteristic $p>0$}
Throughout this section $K$ will be a field of characteristic $p>0$. In this section we prove that being Eulerian is preserved under localization. The proof is quite different from that in characteristic $0$. We also show that each graded $F$-module is always an Eulerian graded $\scr{D}$-module, which will enable us to recover the main result in \cite{YiZhangGradedFModules} in Section \ref{local cohomology}.

\begin{proposition}
\label{localization of Eulerian D-modules in char p} If $M$ is an Eulerian
graded $\scr{D}$-module, so is $S^{-1}M$ for each homogeneous multiplicative system $S\subseteq R$. In particular, $M_f$ is Eulerian for each homogeneous
polynomial $f \in R$.
\end{proposition}
\begin{proof}
First notice that, $\partial_i^{[j]}$ is $R^{p^e}$-linear if $p^e\geq j+1$. So we have
\[\partial_i^{[j]}(z)=\partial_i^{[j]}(f^{p^e}\cdot\frac{z}{f^{p^e}})=f^{p^e}\cdot\partial_i^{[j]}(\frac{z}{f^{p^e}})\]
This tells us that, if $p^e\geq r+1$ and $f\in S$, then $\D\partial_i^{[j]}(\frac{z}{f^{p^e}})=\frac{1}{f^{p^e}}\partial_i^{[j]}(z)$ for every $j\leq r$ in $S^{-1}M$. In particular we have
\[E_r\cdot\frac{z}{f^{p^e}}=\frac{1}{f^{p^e}}E_r\cdot z \]
For any homogeneous $\D\frac{z}{f^t}\in S^{-1}M$, we can multiply both the numerator
and denominator by a large power of $f$ and write $\D\frac{z}{f^t}=\frac{f^{p^e-t}z}{f^{p^e}}$
for some $p^e\geq \max\{r+1,t\}$. So we have
\begin{eqnarray*}
E_r\cdot\frac{z}{f^t}&=&E_r\cdot\frac{f^{p^e-t}z}{f^{p^e}}=\frac{1}{f^{p^e}}E_r\cdot f^{p^e-t}z\\
&=&\frac{1}{f^{p^e}}\binom{\deg(f^{p^e-t})+\deg(z)}{r}\cdot f^{p^e-t}z\\
&=&\binom{p^e\cdot\deg(f)-\deg(f^t)+\deg(z)}{r}\cdot\frac{f^{p^e-t}z}{f^{p^e}}\\
&=&\binom{\deg(\D\frac{z}{f^t})}{r}\cdot\frac{z}{f^t}
\end{eqnarray*}
where the last equality is because $p^e\geq r+1$ and we are in characteristic $p>0$. This finishes the proof.
\end{proof}

Recall the definition of a graded $F$-module as follows.

\begin{definition}[{\it cf.} Definitions 2.1 and 2.2 in \cite{YiZhangGradedFModules}]
For each integer $e\geq 1$, let ${^eR}$ denote the $R$-module that is the same as $R$ as a left $R$-module and whose right $R$-module structure is given by $r'\cdot r=r^pr'$ for all $r'\in {^eR}$ and $r\in R$. An $F$-module is an $R$-module $M$ equipped with an $R$-module isomorphism $\theta:M\to F(M)={^1R}\otimes_RM$. An $F$-module $(M,\theta)$ is called a graded $F$-module if $M$ is graded and $\theta$ is degree-preserving.
\end{definition}

\begin{remark}
It is clear from the definition that, if $(M,\theta)$ is an $F$-module, the map \[\alpha_e:M\xrightarrow{\theta}F(M)\xrightarrow{F(\theta)}F^2(M)\xrightarrow{F^2(\theta)} \cdots \to F^e(M)\]
induced by $\theta$ is also an isomorphism.

This induces a $\scr{D}$-module structure on $M$. To specify the induced $\scr{D}$-module structure, it suffices to specify how $\partial^{[i_1]}_1\cdots \partial^{[i_n]}_n$ acts on $M$. Choose $e$ such that $p^e\geq (i_1+\cdots +i_n)+1$. Given each element $z$, we consider $\alpha_e(z)$ and we will write it as $\sum y_j\otimes z_j$ with $y_j\in {^eR}$ and $z_j\in M$. And we define
\[\partial^{[i_1]}_1\cdots \partial^{[i_n]}_n z:=\alpha_e^{-1}(\sum \partial^{[i_1]}_1\cdots \partial^{[i_n]}_n y_j\otimes z_j).\]
See \cite[\S 1]{KatzmanLyubeznikZhangTwoExamples} for more details.
\end{remark}

When an $F$-module $(M,\theta)$ is graded, the induced map $\alpha_e$ is also degree-preserving. And hence $M$ is naturally a graded $\scr{D}$-module. It turns out that each graded $F$-module is Eulerian as a graded $\scr{D}$-module.

\begin{theorem}
If $M$ is a graded $F$-module, then $M$ is Eulerian graded as a $\scr{D}$-module.
\end{theorem}
\begin{proof}
Pick any homogeneous element $z\in M$, we want to show $E_r\cdot z=\D\binom{\deg(z)}{r}\cdot z$ for each $r\geq 1$. Pick $e$ such that $p^e\geq r+1$. Since $M$ is a graded $F$-module, we have a degree-preserving isomorphism $M\xrightarrow{\alpha_e}F_R^e(M)$. Assume $\alpha_e(z)=\sum_iy_i\otimes z_i$ where $y_i\in R$ and $z_i\in M$ are homogeneous and $\deg(z)=\deg(y_i\otimes z_i)=p^e\deg(z_i)+\deg(y_i)$ for each $i$. In particular we have $\D\binom{\deg(y_i)}{r}=\binom{\deg(z)}{r}$ for every $i$ (because we are in characteristic $p>0$). So we know
\begin{eqnarray*}
E_r\cdot z&=& \alpha_e^{-1}(\sum_i (E_r\cdot y_i)\otimes z_i)\\
&=&\alpha_e^{-1}(\sum_i\binom{\deg(y_i)}{r}y_i\otimes z_i)\\
&=&\alpha_e^{-1}(\binom{\deg(z)}{r}\sum_i y_i\otimes z_i)\\
&=&\binom{\deg(z)}{r}\cdot z
\end{eqnarray*}
This finishes the proof.
\end{proof}

\section{An application to local cohomology}
\label{local cohomology}
Let $R$ be an arbitrary commutative Noetherian ring and $I$ be an ideal of $R$. We recall that if $I$ is generated by
$f_1,\dots,f_l\in R$ and $M$ is any $R$-module, we have the \v{C}ech
complex:
\begin{equation*}
0\rightarrow M\rightarrow \bigoplus_jM_{f_j}\rightarrow
\bigoplus_{j,k}M_{f_jf_k}\rightarrow \dots \rightarrow M_{f_1\cdots
f_l}\rightarrow0
\end{equation*}
whose $i$-th cohomology module is $H_I^i(M)$. Here the map
$M_{f_{j_1}\cdots f_{j_i}} \rightarrow M_{f_{k_1}\cdots
f_{k_{i+1}}}$ induced by the corresponding differential is the
natural localization (up to sign) if $\{ j_1,\dots ,j_i\}$ is a
subset of $\{ k_1,\dots,k_{i+1}\}$ and is $0$ otherwise.

When $R$ is graded and $I$ is a homogeneous ideal
(i.e.$f_1,\dots,f_l$ are homogeneous elements in $R$) and $M$ is a
graded $R$-module, each differential in the \v{C}ech complex is
degree-preserving because natural localization is so. It follows
that each cohomology module
$H^{i_1}_{J_1}(\cdots(H^{i_s}_{J_s}(R)))$ is a graded $R$-module.

When $R=K[x_1,\dots,x_n]$ with $K$ a field of characteristic $p>0$ and $\bm=(x_1,\dots,x_n)$, it is proven in \cite{YiZhangGradedFModules} (using the theory of graded $F$-modules) that

\begin{theorem}[Theorem 3.4 in \cite{YiZhangGradedFModules}]
Let $R=K[x_1,\dots,x_n]$ be a polynomial ring over a field $K$ of
characteristic $p>0$ and $J_1,\dots,J_s$ be homogeneous ideals of
$R$. Each local cohomology module
$H^{i_0}_{\bm}(H^{i_1}_{J_1}\cdots(H^{i_s}_{J_s}(R)))$ is isomorphic
to a direct sum of copies of $\sideset{^*}{}\E(n)$\footnote{when
$s=1$, this is also proved in \cite[page
615]{WenliangZhangLyubeznikNumberProjectiveScheme}} (i.e. all socle
elements of $H^{i_0}_{\bm}(H^{i_1}_{J_1}\cdots(H^{i_s}_{J_s}(R)))$
must have degree $-n$).
\end{theorem}

It is a natural question (and is asked in \cite{YiZhangGradedFModules}) whether the same result holds in characteristic 0. Using our theory of Eulerian graded $\scr{D}$-module, we can give a characteristic-free proof of the same result. In particular we answer the question in characteristic 0 in the affirmative.

We begin with the following easy observation.

\begin{proposition}
Let $J_1,\dots, J_s$ be homogeneous ideals of $R$, then each local cohomology module
$H^{i_1}_{J_1}(\cdots(H^{i_s}_{J_s}(R)))$ is a graded $\scr{D}$-module.
\end{proposition}
\begin{proof}
Since natural localization map is $\scr{D}$-linear (and so is each
differential in the \v{C}ech complex), our proposition follows immediately from
the \v{C}ech complex characterization of local cohomology.
\end{proof}

\begin{theorem}
\label{local cohomology modules are Eulerian}
Let $J_1,\dots, J_s$
be homogeneous ideals of $R$, then each local cohomology module
$H^{i_1}_{J_1}(\cdots(H^{i_s}_{J_s}(R)))$ (considered as a graded
$\scr{D}$-module) is Eulerian.
\end{theorem}
\begin{proof}
This follows immediately from Propositions \ref{localization of Eulerian D-modules in char 0}, \ref{localization of Eulerian D-modules in char p} and \ref{submodules and quotients}, and the \v{C}ech complex characterization of local cohomology.
\end{proof}

\begin{proposition}[{\it cf.} Proposition 2.3 in \cite{LyubeznikFinitenessLocalCohomology} in characteristic $0$ and Lemma (b) on page 208 in \cite{LyubeznikcharacteristicfreeoninjectivedimensionofDmodule} in characteristic $p>0$]
\label{D-module structure E} Let $R=K[x_1,\dots,x_n]$ and $\m=(x_1,\dots,x_n)$. There is a degree-preserving
isomorphism \[\scr{D}/\scr{D}\bm\to \sideset{^*}{}\E.\]
\end{proposition}
\begin{proof}
It is proven in Proposition 2.3 in \cite{LyubeznikFinitenessLocalCohomology} in characteristic $0$ and Lemma (b) on page 208 in \cite{LyubeznikcharacteristicfreeoninjectivedimensionofDmodule} in characteristic $p>0$ that the map $\scr{D}/\scr{D}\bm\to \sideset{^*}{}\E$ given by
\begin{equation}
\label{isomorphism injective hull}
\partial^{[i_1]}_1\cdots \partial^{[i_n]}_n\mapsto (-1)^{i_1+\cdots +i_n}x^{-i_1-1}_1\cdots x^{-i_n-1}_n
\end{equation}
is an isomorphism. The grading on $\scr{D}/\scr{D}\bm$ is induced by
the one on $\scr{D}$ and hence $\deg(\partial^{[i_1]}_1\cdots
\partial^{[i_n]}_n)=-(i_1+\cdots+i_n)$. Since in $\sideset{^*}{}\E$, the socle element $x^{-1}_1\cdots x^{-1}_n$ has degree $0$, it follows that
$\deg(x^{-i_1-1}_1\cdots x^{-i_n-1}_n)=-(i_1+\cdots+i_n)$. Therefore
it follows that  (\ref{isomorphism injective hull}) defines a
degree-preserving isomorphism $\scr{D}/\scr{D}\bm\to
\sideset{^*}{}\E$.
\end{proof}

\begin{proposition}[{\it cf.} Theorem 2.4(a) in \cite{LyubeznikFinitenessLocalCohomology} in characteristic $0$ and Lemma (c) on page 208 in \cite{LyubeznikcharacteristicfreeoninjectivedimensionofDmodule} in characteristic $p>0$]
\label{D-module structure support at m} Let $M$ be a graded
$\scr{D}$-module. If $\Supp_R(M)=\{\bm\}$, then as a graded
$\scr{D}$-module $M\cong \bigoplus_j
\frac{\scr{D}}{\scr{D}\bm}(n_j)\cong
\bigoplus_j\sideset{^*}{}\E(n_j)$.
\end{proposition}
\begin{proof}
$M$ is a graded $\scr{D}$-module hence also graded as an $R$-module.
We first claim that the socle of $M$ can be generated by homogeneous
elements and we reason as follows. Pick a generator $g$ of the socle, we can write it as a
sum of homogeneous elements $g=\sum_{i=1}^tg_i$ where each $g_i$ has
a different degree. For every $x_j\in \bm$, we have
$\sum_{i=1}^tx_j\cdot g_i= x_j\cdot g=0$ (since $g$ is killed by
$\bm$), hence $x_j\cdot g_i=0$ for every $i$ (because each $x_j\cdot
g_i$ has a different degree). Therefore $g_i$ is killed by every
$x_j$, hence is killed by $\bm$, so $g_i$ is in the socle for each
$i$. This proves our claim. We also note that since the socle is killed by $\bm$, a minimal
homogeneous set of generators is actually a homogeneous $K$-basis.

Let $\{e_j \}$ be a homogeneous $K$-basis of the socle of $M$ with
$\deg(e_j)=-n_j$. There is a degree-preserving homomorphism of
$\scr{D}$-modules $\bigoplus_j
\frac{\scr{D}}{\scr{D}\bm}(n_j)\rightarrow M$ which sends $1$ of the
$j$-th copy to $e_j$. This map is injective because it induces an
isomorphism on socles and $\bigoplus_j
\frac{\scr{D}}{\scr{D}\bm}(n_j)$ is supported only at $\bm$ (as an
$R$-module). By \ref{D-module structure E}, $\bigoplus_j
\frac{\scr{D}}{\scr{D}\bm}(n_j)\cong
\bigoplus_j\sideset{^*}{}\E(n_j)$ is an injective $R$-module. So
$ M=\bigoplus_j \frac{\scr{D}}{\scr{D}\bm}(n_j)\bigoplus N$ where $N$
is some graded $R$-module supported only at $\bm$. Since the map on
the socles is an isomorphism, $N=0$, so $M=\bigoplus_j
\frac{\scr{D}}{\scr{D}\bm}(n_j)\cong
\bigoplus_j\sideset{^*}{}\E(n_j)$.
\end{proof}

\begin{theorem}
\label{final theorem}
Let $M$ be an  Eulerian graded $\scr{D}$-module. If $\Supp_R(M)=\{\bm\}$, then $M$ is isomorphic (as a graded $\scr{D}$-module) to a direct sum of copies of $\sideset{^*}{}\E(n)$.
\end{theorem}
\begin{proof}
Since $M$ is supported only at $\bm$, we know it is isomorphic to $\bigoplus_j\sideset{^*}{}\E(n_j)$ as a graded $\scr{D}$-module by Proposition \ref{D-module structure support at m}. By our assumption, $M$ is Eulerian, so is $\bigoplus_j\sideset{^*}{}\E(n_j)$. It follows from  Theorem  \ref{R and E} that $n_j=n$ for each $j$, {\it i.e.} $M$ is isomorphic (as a graded $\scr{D}$-module) to a direct sum of copies of $\sideset{^*}{}\E(n)$. This finishes the proof.
\end{proof}

\begin{corollary}
\label{socle of local cohomology is in degree n}
Let $J_1,\dots, J_s$ be homogeneous ideals of
$R$, then $H^{i_0}_{\bm}H^{i_1}_{J_1}\cdots
H^{i_s}_{J_s}(R)$ is isomorphic (as a graded $\scr{D}$-module) to a
direct sum of copies of $\sideset{^*}{}\E(n)$ (or equivalently, all socle elements of each
$H^{i_0}_{\bm}H^{i_1}_{J_1}\cdots H^{i_s}_{J_s}(R)$ must have degree
$-n$ ).
\end{corollary}
\begin{proof}
This follows immediately from Theorems \ref{local cohomology modules are Eulerian} and \ref{final theorem}.
\end{proof}

\begin{remark}
It is proven in \cite{LyubeznikFinitenessLocalCohomology} (resp, \cite{LyubeznikFModulesApplicationsToLocalCohomology}) that
every $H^{i_1}_{J_1}(\cdots(H^{i_s}_{J_s}(R)))$ is holonomic (resp, $F$-finite) as a
$\scr{D}$-module (resp, $F$-module) in characteristic $0$ (resp, characteristic $p>0$). Therefore in any case we know $H^{i_1}_{J_1}(\cdots(H^{i_s}_{J_s}(R)))$ has finite Bass numbers ({\it cf.} Theorem 3.4(d) in \cite{LyubeznikFinitenessLocalCohomology} and Theorem 2.11 in \cite{LyubeznikFModulesApplicationsToLocalCohomology}). It
follows from this and Corollary \ref{socle of local cohomology is in
degree n} that $H^{i_0}_{\bm}H^{i_1}_{J_1}\cdots
H^{i_s}_{J_s}(R)\cong \sideset{^*}{}\E(n)^c$ for some integer
$c<\infty$.
\end{remark}

\section{Remarks on the graded injective hull of $R/P$ when $P$ is a homogeneous prime ideal}

We have seen in Theorem \ref{R and E} that $\sideset{^*}{}\E(\ell)=\sideset{^*}{}\E(R/\m)(\ell)$ is Eulerian graded if and only if $\ell=n$. In this section we wish to extend this result to $\sideset{^*}{}\E(R/P)$ where $P$ is a non-maximal homogeneous prime ideal (here $\sideset{^*}{}\E(R/P)$ denotes the graded injective hull of $R/P$, see {\it cf.} \cite[Chapter 13.2]{BrodmannSharpLocalCohomology}]). To this end, we will discuss in detail the graded structures of $\sideset{^*}{}\E(R/P)$ as an $R$-module and as a $\scr{D}$-module. The underlying idea is that, there does {\it not} exist a canonical choice of grading on $\sideset{^*}{}\E(R/P)$ when it is considered as a graded $R$-module; however, there is a canonical grading when it is considered as a graded $\scr{D}$-module.

\begin{remark}\label{*E(R/P) as a graded R-module}[$\sideset{^*}{}\E(R/P)$ as a graded $R$-module]
Since $P\neq \mathfrak{m}$, there is at least one $x_i$ that is not contained in $P$. Hence the multiplication by $x_i$ induces an automorphism on $\sideset{^*}{}\E(R/P)$, and consequently we have a degree-preserving isomorphism
\[\sideset{^*}{}\E(R/P)(-1)\xrightarrow{\cdot x_i}\sideset{^*}{}\E(R/P).\]
It follows immediately that
\[\sideset{^*}{}\E(R/P)\cong \sideset{^*}{}\E(R/P)(m)\]
for each integer $m$ in the category of graded $R$-modules. In other words, we have
\[\sideset{^*}{}\E(R/P)(i)\cong \sideset{^*}{}\E(R/P)(j)\]
for all integers $i$ and $j$.

In some sense, this tells us that $\sideset{^*}{}\E(R/P)$ does not have a canonical grading when considered merely as a graded $R$-module.
\end{remark}

However, as we will see, $\sideset{^*}{}\E(R/P)$ is equipped with a natural Eulerian graded $\scr{D}$-module structure, and from this point of view there is indeed a unique natural grading on $\sideset{^*}{}\E(R/P)$. The $\scr{D}$-module structure on $\sideset{^*}{}\E(R/P)$ is obtained via considering $H_P^{\height P}(R)_{(P)}$ where $(\cdot)_{(P)}$ denotes homogeneous localization with respect to $P$ ({\it i.e.} inverting all homogeneous elements not in $P$), which has a natural grading as follows.

\begin{remark}[Grading on $H_P^{\height P}(R)_{(P)}$]
Choose a set of homogeneous generators $f_1,\dots,f_t$ of $P$ and consider the \v{C}ech complex
\[C^{\bullet}(P):\ 0\to R\to \bigoplus_iR_{f_i}\to \cdots\to \bigoplus_{i_1< i_2<\cdots <i_j}R_{f_{i_1}\cdots f_{i_j}}\to \cdots \to R_{f_1\cdots f_t}\to 0.\]
Then, since each module has a natural grading and each differential is degree-preserving, $H^h(C^{\bullet}(P))$ also has a natural grading ($h=\height P$), hence so is $H^h(C^{\bullet}(P))_{(P)}$. We will identify $H_P^{\height P}(R)_{(P)}$ with $H^h(C^{\bullet}(P))_{(P)}$ with its natural grading.
\end{remark}

\begin{proposition}
\label{graded injective hull and local cohomology}
$\sideset{^*}{}\E(R/P)\cong H_P^{\height P}(R)_{(P)}$ in the category of graded $R$-modules.
\end{proposition}
\begin{proof}
We have a graded injective resolution of $R$ (or a *-injective resolution of $R$, {\it cf.} \cite[Chapter 13]{BrodmannSharpLocalCohomology})
\begin{equation}
0\rightarrow R\rightarrow \sideset{^*}{}\E(R)(d_{0})\rightarrow\cdots\rightarrow \bigoplus_{\height Q=s}\sideset{^*}{}\E(R/Q)(d_Q)\rightarrow\cdots\rightarrow\sideset{^*}{}\E(R/\m)(d_\m)\rightarrow 0
\end{equation}
where each $d_{Q}$ is an integer depending on $Q$. Notice that when $Q\neq\m$, $\sideset{^*}{}\E(R/Q)(i)\cong \sideset{^*}{}\E(R/Q)(j)$
for all integers $i$ and $j$ by Remark \ref{*E(R/P) as a graded R-module}. So the above resolution can be written as
\begin{equation}
\label{*-injective resolution of R}
0\rightarrow R\rightarrow \sideset{^*}{}\E(R)\rightarrow\cdots\rightarrow\bigoplus_{\height Q=s}\sideset{^*}{}\E(R/Q)\rightarrow\cdots\rightarrow\sideset{^*}{}\E(R/\m)(d_\m)\rightarrow 0
\end{equation}

Let $h=\height P$. Then $H_P^{h}(R)_{(P)}$ is the homogeneous localization of the $h$-th homology of (\ref{*-injective resolution of R}) when we apply $\Gamma_P(\cdot)$. But when we apply $\Gamma_P(\cdot)$, (\ref{*-injective resolution of R}) becomes
\begin{equation}
\label{*-injective after apply Gamma_P}
0\rightarrow 0\rightarrow \cdots\rightarrow \sideset{^*}{}\E(R/P)\rightarrow\Gamma_P(\bigoplus_{\height Q=h+1}\sideset{^*}{}\E(R/Q))\rightarrow\cdots\rightarrow\Gamma_P(\sideset{^*}{}\E(R/\m)(d_\m))\rightarrow 0
\end{equation}
and when we do homogeneous localization at $P$ to (\ref{*-injective after apply Gamma_P}), we get
\[0\rightarrow \cdots\rightarrow 0\rightarrow \sideset{^*}{}\E(R/P)\rightarrow 0\rightarrow \cdots\rightarrow 0\]
So the $h$-th homology is exactly $\sideset{^*}{}\E(R/P)$ (and when $P=\m$, we get $H_\m^n(R)\cong \sideset{^*}{}\E(R/\m)(d_\m)$, so actually $d_\m=n$). This finishes the proof.
\end{proof}

\begin{remark}[$\scr{D}$-module structure on $\sideset{^*}{}\E(R/P)$]
Since $H^h_P(R)_{(P)}$ has a natural graded $\scr{D}$-module structure, it follows from Proposition \ref{graded injective hull and local cohomology} that $\sideset{^*}{}\E(R/P)$ also has a natural graded $\scr{D}$-module structure.
\end{remark}

Since $\sideset{^*}{}\E(R/P)$ is a graded $\scr{D}$-module, it is natural to ask the following question.

\begin{question}
Let $P$ be a homogeneous prime ideal in $R$. Is there a natural grading on $\sideset{^*}{}\E(R/P)$ making it Eulerian graded?
\end{question}


\begin{remark}
$H_P^{\height P}(R)_{(P)}$ is always Eulerian graded by Theorem \ref{local cohomology modules are Eulerian}, Propositions \ref{localization of Eulerian D-modules in char 0}, and \ref{localization of Eulerian D-modules in char p}. From Remark \ref{basic remarks}(1), we know that, in the category of graded $\scr{D}$-modules, $\sideset{^*}{}\E(R/P)(\ell)$ is Eulerian graded for exactly one $\ell$, we will identify this ``canonical" $\ell$.

Contrary to the case when we consider $\sideset{^*}{}\E(R/P)$ as a graded $R$-module, we can see that in the category of graded $\scr{D}$-modules we have
\[\sideset{^*}{}\E(R/P)(i) \cong \sideset{^*}{}\E(R/P)(j)\ {\rm if\ and\ only\ if\ }i=j.\]
(Otherwise we would have $\sideset{^*}{}\E(R/P)(\ell)\cong \sideset{^*}{}\E(R/P)(\ell+j-i)$ for every $\ell$, and hence there would be more than one choice of $\ell$ such that $\sideset{^*}{}\E(R/P)(\ell)$ is Eulerian.)
\end{remark}

We wish to find the natural grading on $\sideset{^*}{}\E(R/P)$ that makes it Eulerian, and we need the following lemma (which may be well-known to experts).

\begin{lemma}
\label{isomorphism of Ext(R/P, R) and Hom(R/P, H_P(R))}
We have a canonical degree-preserving isomorphism ($h=\height P$)
\[\Ext_R^h(R/P, R)\cong \Hom_R(R/P, H_P^h(R))\]
\end{lemma}
\begin{proof}
$H_P^h(R)$ is the $h$-th homology of (\ref{*-injective resolution of R}) when we apply $\Gamma_P(\cdot)$, which is the $h$-th homology of (\ref{*-injective after apply Gamma_P}), which is the kernel of $\sideset{^*}{}\E(R/P)\rightarrow\Gamma_P(\oplus_{\height Q=h+1}\sideset{^*}{}\E(R/Q))$. Since $\Hom_R(R/P, \cdot)$ is left exact, we know that $\Hom_R(R/P, H_P^h(R))$ is isomorphic to the kernel of \[\sideset{^*}{}\E(R/P)\rightarrow\Hom_R(R/P, \Gamma_P(\oplus_{\height Q=h+1}\sideset{^*}{}\E(R/Q)))\cong \Hom_R(R/P, \oplus_{\height Q=h+1}\sideset{^*}{}\E(R/Q)))\] But this is exactly the $h$-th homology of (\ref{*-injective resolution of R}) when we apply $\Hom_R(R/P, \cdot)$, which by definition is $\Ext_R^h(R/P, R)$. And we want to emphasize here that the isomorphism obtained does not depend on the grading on $\sideset{^*}{}\E(R/P)$ as long as $\sideset{^*}{}\E(R/P)$ is equipped with the same grading when we calculate $\Ext_R^h(R/P, R)$ and $\Hom_R(R/P, H_P^h(R))$ as above.
\end{proof}

\begin{definition}
For a $d$-dimensional graded $K$-algebra $S$ with irrelevant maximal ideal $\m$, the {\it a-invariant} of $S$ is defined to be
\[a(S)=\max\{t\in\mathbb{Z}|H_\m^d(S)_t\neq 0\}.\]
\end{definition}

\begin{proposition}
\label{canonical inclusion}
We have $\min\{t|(\Ann_{H_P^{\height P}(R)}P)_t\neq 0\}=-a(R/P)-n$. Hence we have a degree-preserving inclusion
\[\D R/P\hookrightarrow H_P^{\height P}(R)(-a(R/P)-n)\hookrightarrow H_P^{\height P}(R)_{(P)}(-a(R/P)-n).\]
\end{proposition}
\begin{proof}
Let $h=\height P$ and let $s=\min\{t|(\Ann_{H_P^{\height P}(R)}P)_t\neq 0\}$. By lemma \ref{isomorphism of Ext(R/P, R) and Hom(R/P, H_P(R))}, we know
\[\Ext_R^h(R/P, R)\cong \Hom_R(R/P, H_P^h(R))\cong \Ann_{H_P^{h}(R)}P\]
so we know
\[s=\min\{t|(\Ext_R^h(R/P, R)_t\neq 0\}=\min\{t|(\Ext_R^h(R/P, R(-n))_t\neq 0\}-n\]
by graded local duality
\[\min\{t|(\Ext_R^h(R/P, R(-n))_t\neq 0\}=-\max\{t\in\mathbb{Z}|H_\m^{n-h}(R/P)_t\neq 0\}=-a(R/P)\]
Hence we get $s=-a(R/P)-n$. The second statement follows from the first one by sending $\overline{1}$ in $R/P$ to any element in $\Ann_{H_P^{\height P}(R)}P$ of degree $-a(R/P)-n$.
\end{proof}

\begin{remark}
\label{grading on injective hull}
From what we have discussed so far, we can see that $H_P^{\height P}(R)_{(P)}$ is a graded injective module (or *-injective module) and there is a degree-preserving inclusion $R/P\hookrightarrow H_P^{\height P}(R)_{(P)}(-a(R/P)-n)$, always sending $\overline{1}$ in $R/P$ to the lowest degree element in $\Ann_{H_P^{\height P}(R)}P$. Therefore we propose a ``canonical" grading on $\sideset{^*}{}\E(R/P)$ to the effect that $\sideset{^*}{}\E(R/P)$ can be identified with $H_P^{\height P}(R)_{(P)}(-a(R/P)-n)$ (where the grading on $H_P^{\height P}(R)_{(P)}$ is obtained via \v{C}ech complex).
\end{remark}

We end with the following proposition.
\begin{proposition}[Compare with Theorem 13.2.10 and Lemma 13.3.3 in \cite{BrodmannSharpLocalCohomology}]
Given the grading on $\sideset{^*}{}\E(R/P)$ as proposed in Remark \ref{grading on injective hull}, we have that
\begin{enumerate}
\item $\sideset{^*}{}\E(R/P)(\ell)$ is Eulerian if and only if $\ell=a(R/P)+n$;
\item the minimal graded injective resolution (or *-injective resolution) of $R$ can be written as
\[0\to R\to \sideset{^*}{}\E(R)\to \cdots\to\bigoplus_{\height P=j} \sideset{^*}{}\E(R/P)(a(R/P)+n)\to\cdots\to \sideset{^*}{}\E(R/\m)(n)\to 0.\]
\end{enumerate}
\end{proposition}
\begin{proof}
(1). This is clear by our grading on $\sideset{^*}{}\E(R/P)$ and the fact that there is a unique grading on $H_P^{\height P}(R)_{(P)}$ that makes it Eulerian.

(2). This follows immediately from the calculation of  $H_P^{\height P}(R)_{(P)}$ using the minimal *-injective resolution of $R$ (note that $a(R)=-n$ and $a(R/\m)=0$).
\end{proof}

\bibliographystyle{skalpha}
\bibliography{CommonBib}
\end{document}